\documentclass[10pt,a4paper]{article}

\usepackage{amsthm}
\usepackage[round]{natbib}
\usepackage{kashsty}

\usepackage[margin=1.33in]{geometry}

\def\PICDIR{.}

\title{ Improved Rademacher symmetrization through a Wasserstein based measure of asymmetry }
\author{Adam B Kashlak\\Cambridge Centre for Analysis}

\begin{document}

\maketitle

\begin{abstract}
  We propose of an improved version of the ubiquitous symmetrization
  inequality making use of the Wasserstein distance between
  a measure and its reflection in order to quantify the symmetry of
  the given measure.  An empirical bound
  on this asymmetric correction term is derived through a
  bootstrap procedure and shown to give tighter results in
  practical settings than the original uncorrected inequality.
  Lastly, a wide range of applications are detailed including
  testing for data symmetry,
  constructing nonasymptotic high dimensional confidence sets,
  bounding the variance of an empirical process, 
  and improving constants in Nemirovski style inequalities
  for Banach space valued random variables.
\end{abstract}

\section{Introduction}

The symmetrization inequality is a ubiquitous result
in the probability in Banach spaces literature and 
in the concentration of measure literature.  
Dating back at least to Paul L{\'e}vy, 
it is found in the classic text of 
\cite{LEDOUXTALAGRAND1991}, Section 6.1, and the more recent
\cite{BOUCHERON2013}, Section 11.3.
\cite{GINEZINN1984} use symmetrization
in the context of empirical process theory, which is 
followed by a collection of more recent appearances such 
as 
\cite{PANCHENKO2003,KOLTCHINSKII2006,GINENICKL2010ADAPTIVE,
ARLOT2010,LOUNICINICKL2011,KERKYACHARIAN2012,FANPARTIII}.

The symmetrization inequality is as follows.  
Let $(B,\norm{\cdot})$ be a Banach space, and let $X_1,\ldots,X_n\in B$
be \iid random variables with measure $\mu$.  
Let $\veps_1,\ldots,\veps_n$ be
\iid Rademacher random variables, which are such that 
$\prob{\veps_i=1}=\prob{\veps_i=-1}=1/2$.  
These are sometimes referred to as symmetric Bernoulli
or random signs.
The symmetrization inequality is
$$
  \xv\norm*{ \frac{1}{n}\sum_{i=1}^n(X_i-\xv X_i) } \le
  2 \xv\norm*{ \frac{1}{n}\sum_{i=1}^n\veps_i(X_i-\xv X_i) }.
$$
This can be readily proved via Jensen's Inequality and the
insight that if $Z$ is a symmetric random variable, that is 
$Z\eqdist -Z$, then $Z\eqdist\veps Z$.  

The most notable
oversight of this result is that it does not incorporate 
any measure of the symmetry of the data. 
Specifically, in the extreme case that the $X_i$ are symmetric
about their mean, then the coefficient of 2 can be dropped
and the inequality becomes an equality.  
Taking note of this fact, \cite{ARLOT2010} state that
``it can be shown that this factor of 2 is unavoidable 
in general for a fixed $n$ when
the symmetry assumption is not satisfied, although it is 
unnecessary when $n$ goes to infinity.'' 
They furthermore ``conjecture that 
an inequality holds under an assumption less restrictive
than symmetry (e.g., concerning an appropriate measure of 
skewness of the distribution ).''
Hence, in response to this conjecture, 
we propose an improved symmetrization
inequality making use of Wasserstein distance and Hilbert space
geometry in order
to account for the symmetry, or lack thereof,
of the distribution of the $X_i$ under analysis.  
The main contribution of this paper is that
for some Hilbert space $H$ and $X_1,\ldots,X_n\in H$ iid random variables
with measure $\mu$,
there is for a fixed constant $C(\mu)$ depending only on the 
symmetry of the 
underlying measure $\mu$ of the $X_i$, which quantifies the 
symmetry of $\mu$, such that
$$
  \xv\norm*{ \frac{1}{n}\sum_{i=1}^n(X_i-\xv X_i) } \le
  \xv\norm*{ \frac{1}{n}\sum_{i=1}^n\veps_i(X_i-\xv X_i) }
  + \frac{C(\mu)}{n^{1/2}}.
$$
This result is detailed and proved in Section~\ref{sec:symres}.
Furthermore, an empirical bound, $C_n(\mu)$, on the constant $C$
can be calculated as is done in Section~\ref{sec:symcomp}.  
In the case that the distribution of the $X_i$ 
is symmetric, our data driven 
estimate $C_n(X)=O(n^{-1/2})$ implying an $n^{-1}$ 
rate of convergence to the desired zero for the additive 
term above.
Applications of this result 
to testing the symmetry of a data set, constructing nonasymptotic 
high dimensional
confidence sets, bounding the variance of an empirical process,
and improving coefficients in probabilistic inequalities in the 
Banach space setting
are given in Section~\ref{sec:symapps}.

%This result is detailed in 
%Section~\ref{sec:symproof} with information on its computation
%in Section~\ref{sec:symcomp} and with some applications presented
%in Section~\ref{sec:symapps}.

\section{Symmetrization}
\label{sec:symproof}

\subsection{Definitions}
\label{sec:symdef}

We first require the standard notions of Wasserstein distance
and Wasserstein space as stated below.  For a thorough introduction
to such topics, see~\cite{VILLANI2008OPTTRANS}. 

\begin{defn}[Wasserstein Distance]
  \label{def:wasser}
  Let $(\mathcal{X},d)$ be a Polish space and $p\in[1,\infty)$.
  For two probability measures $\mu$ and $\nu$ on $\mathcal{X}$,
  the Wasserstein $p$ distance is
  $$
    W_p(\mu,\nu) = \inf_{\gamma\in\Pi(\mu,\nu)} \left(
      \int_{\mathcal{X}\times\mathcal{X}} d(x,y)^p d\gamma(x,y)
    \right)^{1/p}
  $$
  where the infimum is taken over all measures $\gamma$
  on $\mathcal{X}\times\mathcal{X}$ with marginals~$\mu$
  and~$\nu$.
\end{defn}

An equivalent and useful formulation of Wasserstein distance is 
$$
    W_p(\mu,\nu) = \inf_{(X,Y)} \left(
      \xv\, d(X,Y)^p 
    \right)^{1/p}
$$
where the infimum is taken over all possible joint distributions
of $X$ and $Y$ with marginals $\mu$ and $\nu$, respectively.

\begin{defn}[Wasserstein Space]
  \label{def:wasserSpace}
  Let $P(\mathcal{X})$ be the space of probability measures on 
  $\mathcal{X}$.  The Wasserstein space is 
  $$
    P_p(\mathcal{X}) := \left\{
      \mu\in P(\mathcal{X}) \,\middle|\,
      \int_\mathcal{X} d(x_0,x)^p \mu(dx) < \infty
    \right\}
  $$
  for any arbitrary choice of $x_0$.
  This is the space of measures with finite $p$th moment.
\end{defn}

Convergence in Wasserstein space is characterized by 
weak convergence of measure and convergence in $p$th moment.
From Theorem~6.8 of \cite{VILLANI2008OPTTRANS},
convergence in Wasserstein distance is equivalent to 
weak convergence in $P_p(\mathcal{X})$.
Hence, for a sequence of measures
$\mu_n$,
$$
  W_p(\mu_n,\mu) \rightarrow 0 \text{ \iff }
  \mu_n\convd\mu \text{ and } 
  \int_{\mathcal{X}} x^p d\mu_n(x) \rightarrow
  \int_{\mathcal{X}} x^p d\mu(x).
$$

Secondly, we will make use of empirical measures
and the already mentioned Rademacher random variables.

\begin{defn}[Empirical Measure]
  For \iid random variables $X_1,\ldots,X_n$, the empirical
  measure is a random measure defined as
  $$
    \mu_n(A) := \frac{1}{n}\sum_{i=1}^n \indc{X_i\in A}
  $$
  for some measurable set $A$.  We will denote the empirical
  measure of the reflected variables $-X_1,\ldots,-X_n$ by $\mu_n^{-}$.
\end{defn}

\begin{defn}[Rademacher Distribution]
  A random variable $\veps\in\real$ has a 
  {Rademacher} distribution
  if 
  $
    \prob{\veps=1} = \prob{\veps=-1} = 1/2.
  $
  In Section~\ref{sec:symcomp}, we will also consider 
  more general Rademacher$(p)$ distributions where
  $\prob{\veps=1} = p$ and conversely $\prob{\veps=-1}=1-p$.
\end{defn}

\subsection{Symmetrization Result}
\label{sec:symres}

In the following lemma, we bound the expectation on 
the left by the sum of a ``symmetric'' term and an ``asymmetric'' term.

\begin{lm}
  \label{lem:symmet}
  Let $H$ be an Hilbert space, and
  let $X_1,\ldots,X_n\in H$ be \iid random variables
  with common law $\mu$.  Define $\mu^-$ to be the law of $-X$.
  Furthermore, let $\veps_1,\ldots,\veps_n$ be \iid 
  Rademacher random variables also independent of the 
  $X_i$.  Then, for any 1-Lipschitz function $\psi$,
  $$
    \xv \psi\left(
      \sum_{i=1}^n (X_i-\xv X_i) 
    \right) \le
    \xv \psi\left(
      \sum_{i=1}^n \veps_i( X_i-\xv X_i )
    \right) +
    \sqrt{\frac{n}{2}}W_2(\mu,\mu^-)
  $$
  where $W_2$ is the Wasserstein 2 distance.
\end{lm}
\begin{proof}
  For a Polish space $\mathcal{X}$, 
  let $\Pi(\mu,\nu)$ be the space of all product measures on 
  $\mathcal{X}\times\mathcal{X}$ with marginals $\mu$ and
  $\nu$.  
  For $\delta\in(0,1)$, let
  $\Pi_\delta(\mu,\nu)$ be the space of all product measures 
  with marginals $\mu$ and $\nu_\delta=\delta\mu+(1-\delta)\nu$.
  For $\gamma\in\Pi(\mu,\nu)$ and $\eta\in\Pi(\mu,\mu)$, the measure
  $\delta\eta + (1-\delta)\gamma\in\Pi_\delta(\mu,\nu)$.
  Hence,
  \begin{align*}
    W_p^p(\mu,\nu_\delta) 
    &= \inf_{\gamma_\delta\in\Pi(\mu,\nu_\delta)}
     \int_{\mathcal{X}\times\mathcal{X}} d(x,y)^p d\gamma_\delta(x,y) \\
    &\le \inf_{\eta\in\Pi(\mu,\mu),\,\gamma\in\Pi(\mu,\nu)}
     \int_{\mathcal{X}\times\mathcal{X}} d(x,y)^p 
     d(\delta\eta + (1-\delta)\gamma)(x,y) \\
    &= \inf_{\gamma\in{\Pi}(\mu,\nu)}
     (1-\delta)\int_{\mathcal{X}\times\mathcal{X}} d(x,y)^p 
     d\gamma(x,y) \\
    &= (1-\delta)W_p^p(\mu,\nu).
  \end{align*}
  The inequality on the second lines above arises from taking 
  the infimum over a more restrictive set.
  The law of $\veps X$ is $\frac{1}{2}(\mu+\mu^-)$. 
  Hence, for our purposes, the above
  implies that 
  $$
    W_2\left(\mu,\frac{\mu+\mu^-}{2}\right) \le 
    \frac{1}{\sqrt{2}} W_2(\mu,\mu^-).
  $$

  Define $\mu^{*n}$ to be the law of 
  $\sum_{i=1}^n (X_i-\xv X_i)$ and 
  $\tilde{\mu}^{*n}$ to be the law of 
  $\sum_{i=1}^n \veps_i(X_i-\xv X_i)$. Then,
  \begin{flalign*}
    &~~~~\xv \psi\left(
      \sum_{i=1}^n (X_i-\xv X_i) 
    \right) -
    \xv \psi\left(
      \sum_{i=1}^n \veps_i( X_i-\xv X_i )
    \right) \le&
  \end{flalign*}
  \vspace{-0.2in}
  \begin{align*}
    &\le \sup_{\norm{\phi}_{Lip}\le 1}\left\{
      \xv \phi\left(
        \sum_{i=1}^n (X_i-\xv X_i) 
      \right) -
      \xv \phi\left(
        \sum_{i=1}^n \veps_i( X_i-\xv X_i )
      \right) 
    \right\} \\
    &\le W_1\left( \mu^{*n}, \tilde{\mu}^{*n} \right) \\
    &\le W_2\left( \mu^{*n}, \tilde{\mu}^{*n} \right) \\
    &\le \sqrt{n} W_2\left( \mu, \frac{\mu+\mu}{2}^- \right)\\
    &\le \sqrt{\frac{n}{2}} W_2( \mu, \mu^- )
  \end{align*}
  where the second, third, and fourth inequality come respectively 
  from Lemmas~\ref{lem:duality}, \ref{lem:order}, and~\ref{lem:conv}
  in the appendix.
  Rearranging the terms gives the desired result.
\end{proof}

This lemma leads immediately to the following theorem.  
The intuition behind this theorem is that averaging a 
collection of random variables has an inherent smoothing
and symmetrizing effect.  Thus, as the sample size $n$ 
increases, the difference between the expectations of the
true average and the Rademacher average become negligible.

\begin{thm}
  \label{thm:symovern}
  Using the setting of Lemma~\ref{lem:symmet} 
  with either of the following two
  conditions that
  \begin{enumerate}
    \item $\psi$ is additionally positive homogeneous (e.g. a norm), or
    \item the metric $d$ is positive homogeneous 
          in the sense that for $a\in\real$,
          $d(ax,ay) = \abs{a}d(x,y)$,
  \end{enumerate}
  then
  $$
    \abs*{
    \xv \psi\left(
      \frac{1}{n}\sum_{i=1}^n (X_i-\xv X_i) 
    \right) -
    \xv \psi\left(
      \frac{1}{n}\sum_{i=1}^n \veps_i( X_i-\xv X_i )
    \right)
    } = O\left(\frac{1}{\sqrt{n}}\right).
  $$
\end{thm}
\begin{proof}
  Running the 
  proof of Lemma~\ref{lem:symmet} after swapping
  $\sum_{i=1}^n (X_i-\xv X_i)$ and $\sum_{i=1}^n \veps_i( X_i-\xv X_i )$
  gives the lower deviation
  $$
    \xv \psi\left(
      \sum_{i=1}^n (X_i-\xv X_i) 
    \right) \ge
    \xv \psi\left(
      \sum_{i=1}^n \veps_i( X_i-\xv X_i )
    \right) -
    \sqrt{\frac{n}{2}}W_2(\mu,\mu^-).
  $$
  Under condition~1, the result is immediate. 
 
  Under condition~2,
  let $\mu$ be the law of $(X_i-\xv X_i)$ as before.
  Then, redefining
  $\mu^{*n}$ to be the law of 
  $\sum_{i=1}^n \frac{1}{n}(X_i-\xv X_i)$ and 
  $\tilde{\mu}^{*n}$ to be the law of 
  $\sum_{i=1}^n \frac{1}{n}\veps_i(X_i-\xv X_i)$ results in
  \begin{align*}
    W_2(\mu^{*n},\tilde{\mu}^{*n}) 
    &\le \sqrt{n}\inf_{(X,Y)}\left(
      \xv\, d(X/n,Y/n)^2
    \right)^{1/2} \\
    &= \frac{1}{\sqrt{2n}}W_2(\mu,\mu^-)
  \end{align*}
  where the infimum is taken over all joint distributions 
  of $X$ and $Y$ with marginals $\mu$ and $\frac{\mu+\mu^-}{2}$,
  respectively.  The desired result follows.
\end{proof}

%\begin{rk}
%  Perhaps the rate of convergence can be quantified.
%  I have yet to read del Barrio, Gin{\'e}, and 
%  Matr{\'a}n~\cite{BARRIOGINEMATRAN1999} who, based on their
%  abstract, may discuss this point.
%\end{rk}
%\begin{rk}
%  Given observations $X_1,\ldots,X_n$, one could presumably
%  estimate $W_2(\mu_n,\mu_n^-)$ via some bootstrap resampling
%  method.
%\end{rk}

%\section{Computation}
%\label{sec:symcomp}

\section{Empirical estimate of $W_2(\mu,\mu^-)$}
\label{sec:symcomp}

In order to explicitly make use of the above results,
an empirical estimate of $W_2(\mu,\mu^-)$ is required.
We first establish the following bound.
\begin{prop}
  \label{prp:empBound}
  Let $X_1,\ldots,X_n$ be iid with law $\mu$ and 
  let $Y_1,\ldots,Y_n$ be iid with law $\nu$.  Furthermore, 
  let $\mu_n$ and $\nu_n$ be the empirical distributions 
  of $\mu$ and $\nu$, respectively.  Then,
  $$
    W_p^p(\mu,\nu) \le \xv W_p^p(\mu_n,\nu_n).
  $$
\end{prop}
\begin{proof}
  The following infima are taken over the
  joint distributions of the random variables in question.
  Let $X$ and $Y$ be random variables of law $\mu$ and $\nu$,
  respectively.  Also, let $S_n$ be the group of permutations
  on $n$ elements.
  \begin{align*}
    W_p^p(\mu,\nu)
    &= \inf_{(X,Y)} \xv d(X,Y)^p \\
    &= \inf_{(X_1,\ldots,X_n,Y_1,\ldots,Y_n)}\xv\left\{
      \frac{1}{n}\sum_{i=1}^n d(X_i,Y_i)^p
    \right\}\\
    &\le \xv\min_{\rho\in S_n}\left\{
      \frac{1}{n}\sum_{i=1}^n d(X_i,Y_{\rho(i)})^p
    \right\}\\
    &= \xv W_p^p(\mu_n,\nu_n)
  \end{align*}
  where the above inequality arises by replacing the infimum over
  all possible joint distributions of the $X_i$ and $Y_i$ 
  with a specific joint distribution.
\end{proof}

The following subsections establish that it is reasonable to
replace $W_2(\mu,\mu^-)$ with a data driven estimate of 
$\xv W_2(\mu_n,\mu_n^-)$ in Lemma~\ref{lem:symmet} 
and Theorem~\ref{thm:symovern}. 
Rates of convergence of $W_2(\mu_n,\mu_n^-)$ are presented, 
and a bootstrap estimator for $\xv W_2(\mu_n,\mu_n^-)$ is
proposed and tested numerically.

\subsection{Rate of convergence of empirical estimate}

As $W_p(\cdot,\cdot)$ is a metric, 
the triangle inequality implies that
\begin{align*}
  W_p(\mu,\mu^-) 
  &\le W_p(\mu,\mu_n) + W_p(\mu_n,\mu_n^-) + W_p(\mu_n^-,\mu^-)\\
  &\le 2W_p(\mu,\mu_n) + W_p(\mu_n,\mu_n^-),
\end{align*}
and therefore,
$$
\abs{W_p(\mu,\mu^-)-W_p(\mu_n,\mu_n^-)} \le 2W_p(\mu,\mu_n).
$$
By Lemma~\ref{lem:converge}, $W_p(\mu,\mu_n)\rightarrow0$ with
probability one making the discrepancy negligible for large data sets.
However, it is also possible 
to get a hard upper bound on this term;
specifically,  the recent work of \cite{FOURNIER2015}
proposes explicit moment bounds on $W_p(\mu,\mu_n)$.
Their result can be used to demonstrate the 
speed with which our empirical measure of asymmetry, 
$W_2(\mu_n,\mu_n^-)$, converges to zero 
when $\mu$ is symmetric.

  In the case that $\mu$ is symmetric, $W_2(\mu,\mu^-)=0$,
  the ideal correction term is equal to zero.
  This implies that our empirical bound 
  $$
    W_2(\mu_n,\mu_n^-)=\abs*{W_2(\mu,\mu^-)-W_2(\mu_n,\mu_n^-)}
    \le 2W_2(\mu,\mu_n).
  $$
  Therefore, the moment bound from Theorem~1 of \cite{FOURNIER2015}
  implies that $W_2(\mu_n,\mu_n^-) = O(n^{-1/2 - \delta})$ where
  $\delta\in(0,0.5]$ depending on the specific moment used and the 
  dimensionality of the measure.
  Thus, the empirical Wasserstein distance 
  achieves a faster convergence rate in the symmetric case than the 
  general rate of $n^{-1/2}$.

The tightness of the bounds proposed in \cite{FOURNIER2015} was
tested experimentally.  While the moment bounds are certainly of 
theoretical interest, implementing these bounds
resulted in an inequality less sharp than the original
symmetrization inequality.
However, the bootstrap procedure detailed in the following section
does produce a practically useful estimate of the expected 
empirical Wasserstein distance.

\subsection{Bootstrap Estimator}

We propose a bootstrap procedure to 
estimate  
the expected Wasserstein distance between the
empirical measure and its reflection, $\xv W_2(\mu_n,\mu_n^-)$.
Given observations $x_1,\ldots,x_n$, let
$\hat{\mu}_n$ be the empirical measure of
the data.  Then, for some specified $m$, 
two sets $Y_1,\ldots,Y_m$ and $Z_1,\ldots,Z_m$
can be sampled as independent draws from $\hat{\mu}_n$.  
The goal is to move a mass of $1/m$ from each of the $Y_i$
to each of the negated $-Z_i$ in an optimal fashion.
Hence, the $m\times m$
matrix of pairwise distances is constructed with entries
$A_{i,j} = d(Y_i,-Z_j)$, which can be accomplished
in $O(m^2)$  time.
From here, the problem reduces to a \textit{linear assignment problem},
a specific instantiation of a \textit{Minimum-cost flow problem} 
from linear programming
\citep{AHUJA1993NETWORKFLOWS}.
That is, given a complete bipartite graph with vertices 
$L\cup R$ such that $|L|=|R|=m$ and with weighted edges,
we wish to construct a perfect matching minimizing the total
sum of the edge weights. Here, the weights are the pairwise distances
$A_{i,j}$.  This linear program can be efficiently solved in 
$O(m^3)$ time via the 
\textit{Hungarian algorithm}~\citep{KUHN1955HUNGARIAN}.  For more
on linear programs in the probabilistic setting, see
\cite{STEELE1997}.

This estimated distance can be averaged over multiple bootstrapped
samples.  Though, in general, only a few replications are necessary
to achieve a stable estimate
as the bootstrap estimator has a very small variance.  Indeed, 
to see this, consider
the bounded difference inequality detailed in Section 3.2 of
\cite{BOUCHERON2013}, which is a direct corollary of the 
Efron-Stein-Steele inequality 
\citep{EFRONSTEIN1981,STEELE1986,RHEETALAGRAND1986}.
\begin{defn}[A function of bounded differences]
  For $\mathcal{X}$ some measurable space and 
  a real valued function $f:\mathcal{X}^n\rightarrow \real$,
  $f$ is said to have the bounded differences property if
  for all $i=1,\ldots,n$,
  $$
    \sup_{x_1,\ldots,x_n,x_i'}\abs{
      f(x_1,\ldots,x_n) - f(x_1,\ldots,x_i',\ldots,x_n)
    } \le c_i.
  $$
\end{defn}

\begin{prop}[Corollary 3.2 of \cite{BOUCHERON2013}]
  If $f$ has the bounded differences property with constants
  $c_1,\ldots,c_n$, then 
  $\var{f(X_1,\ldots,X_n)}\le\frac{1}{4}\sum_{i=1}^nc_i^2$.
\end{prop}

In our setting, $Y_i$ and $Z_i$ for $i=1,\ldots,m$ are 
independent random variables with law $\hat{\mu}_n$.  
The function $f(Y_1,\ldots,Y_m,Z_1,\ldots,Z_m)$ is the 
value of the optimal matching from the $\{Y_i\}$ to the $\{-Z_i\}$.
This $f$ is, in fact, a function of bounded differences, because 
modifying a single argument
will at most change the optimal value by 
$
  c = m^{-1}(
    \max_{i,j=1,\ldots,n}\{d(x_i,-x_j)\} -
    \min_{i,j=1,\ldots,n}\{d(x_i,-x_j)\}
  ) = C/m.
$
Thus, from the bounded differences theorem,
$$
  \var{f(Y_1,\ldots,Y_m,Z_1,\ldots,Z_m)} \le
  \frac{C^2n}{4m^2}.
$$
Therefore, if $m$ is chosen to be of order $n$, as in the 
numerical experiments below, then 
the variance of the bootstrap estimate decays at 
rate of $O(n^{-1})$.

The proposed bootstrap procedure was experimentally tested 
on both high dimensional 
Rademacher and Gaussian data as will be seen in 
Sections~\ref{sec:numRad} and~\ref{sec:numGauss}.  For each replication, 
the observed data was
randomly split in half.  That is, given a random permutation
$\rho \in S_n$, the symmetric group on $n$ elements, 
the Hungarian algorithm was run to calculate the cost of an 
optimal perfect matching
between $\{X_{\rho(1)},\ldots,X_{\rho(\frac{n}{2})}\}$ and
$\{-X_{\rho(\frac{n}{2}+1)},\ldots,-X_{\rho(n)}\}$.

\subsection{Numerical Experiments}

From Proposition~\ref{prp:empBound}, there is an obvious 
positive bias 
in our new symmetrization inequality when using 
the Wasserstein
distance between the empirical measures, $W_2(\mu_n,\mu_n^-)$, 
in lieu of the Wasserstein distance between the 
unknown underlying measures, $W_2(\mu,\mu^-)$.
This is specifically troublesome when $\mu$ is 
symmetric or nearly symmetric.  That is,
if $W_2(\mu,\mu^-)=0$,
then barring trivial cases, 
the distance between the empirical measures will be positive with positive
probability.
However, as stated in Lemma~\ref{lem:converge}, 
$W_2(\mu_n,\mu_n^-)\rightarrow 0$ with probability one,
which will still make this approach superior to the standard
symmetrization inequality.
In the following subsections, we will compare the magnitude of the 
expected symmetrized sum and the asymmetric correction term,
which are, respectively,
$$
  R_n = n^{-1/2}\xv\norm*{ \sum_{i=1}^n \veps_i(X_i-\xv X_i) }
  ~~\text{ and  }~~
  C_n = W_2(\mu_n,\mu_n^-)/\sqrt{2}.
$$
The goal is to demonstrate through numerical simulations
that the latter is smaller than the former and thus that
newly proposed
$R_n+C_n$ is 
a sharper upper bound than the original $2R_n$ for 
$n^{-1/2}\xv\norm*{ \sum_{i=1}^n (X_i-\xv X_i) }$.

\subsubsection{Rademacher Data}
\label{sec:numRad}

For a dimension $k$ and a sample size 
$n=\{2,4,8,\ldots,256\}$, 
the data for this first
numerical test was generated from a multivariate symmetric Rademacher
distribution.  That is,
for a size $n$ iid sample from this distribution, $X_1,\ldots,X_n$,
let $X_{i,j}$ be the $j$th entry of the $i$th random variable 
with $X_{i,1},\ldots,X_{i,k}$ iid Rademacher$(1/2)$ random variables.
Across 10,000 replications, random samples were drawn and
used to estimate the expected Rademacher average, $R_n$, and the
expected empirical Wasserstein distance, $C_n$, under the $\ell_1$-norm. 
The dimensions 
considered were $k=\{2,20,200\}$.  The results are displayed
on the left column of Figure~\ref{fig:simByN}.  As the sample size $n$
increases with respect to $k$, we get closer to an asymptotic 
state and the bound based on the empirical Wasserstein distance
becomes more attractive.

\subsubsection{Gaussian Data}
\label{sec:numGauss}

For a dimension $k$ and a sample size 
$n=\{2,4,8,\ldots,256\}$, 
the data for this second
numerical test was generated from a multivariate Gaussian mixture
distribution.  
Specifically, 
$
  \frac{1}{2}\distNormal{{\bf -1}}{I_k} + 
  \frac{1}{2}\distNormal{{\bf  1}}{I_k},  
$
which is a symmetric distribution.
Over 10,000 replications, random samples were drawn and
used to estimate the expected Rademacher average, $R_n$, and the
expected empirical Wasserstein distance, $C_n$, under the $\ell_2$-norm. 
The dimensions 
considered were $k=\{2,20,200\}$.  The results are displayed
on the right column of Figure~\ref{fig:simByN}.  Similarly to the 
multivariate Rademacher setting, as the sample size $n$
increases, the bound based on the empirical Wasserstein distance
becomes sharper than the original symmetrization bound.

\begin{figure}
  \centering
  \includegraphics[width=\textwidth]{\PICDIR/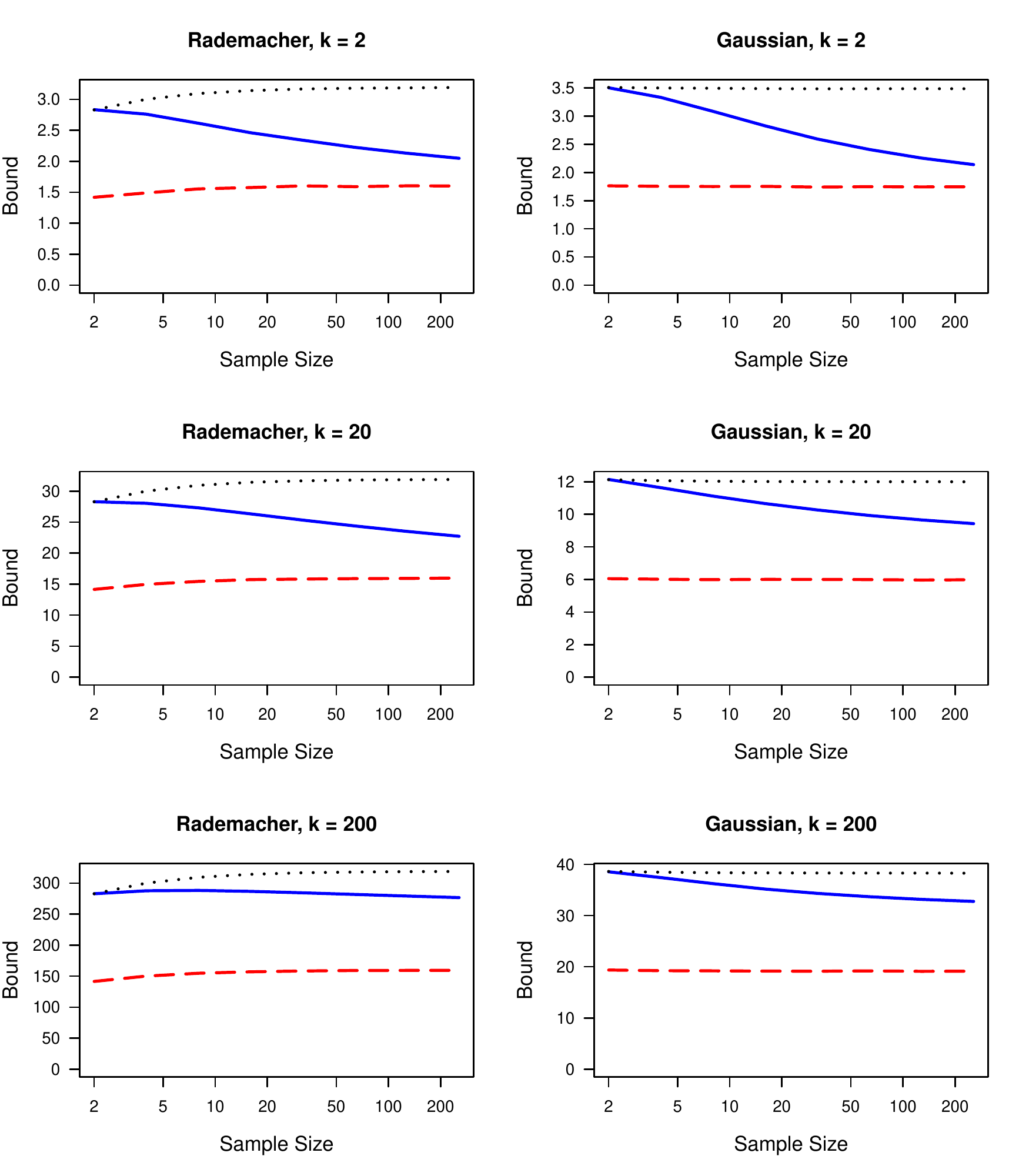}
  \capt{
    \label{fig:simByN}
    For multivariate Rademacher (left) and Gaussian mixture (right)
    data, the average
    $n^{-1/2}\xv\norm{\sum_{i=1}^n(X_i-\xv X_i)}$ (red dashed lines), 
    twice the Rademacher average 
    $2R_n = 2n^{-1/2}\xv\norm{\sum_{i=1}^n\veps_i(X_i-\xv X_i)}$ 
    (black dotted lines), and the bound using the 
    scaled empirical Wasserstein distance, 
    $R_n + W_2(\mu_n,\mu_n^-)/\sqrt{2}$ (blue solid lines) were 
    estimated over 10,000 replications.  The dimension of the
    data is $k=\{2,20,200\}$.  For the Rademacher setting, the
    $\ell_1$-norm was used.  For the Gaussian setting, the 
    $\ell_2$-norm was used.  As the sample size increases,
    the Wasserstein term converges to zero thus sharpening 
    the upper bound.
  }
\end{figure}

\section{Applications}
\label{sec:symapps}

In the following subsections, a collection of applications 
of the improved symmetrization inequality are detailed.
These include a test for data symmetry, the construction
of nonasymptotic high dimensional confidence sets, bounding the variance
of an empirical process, and Nemirovski's inequality for
Banach space valued random variables.

\subsection{Permutation test for data symmetry}

In the previous sections, we proposed the Wasserstein distance 
$W_2(\mu,\mu^-)$ to quantify the symmetry of a measure $\mu$.
Now, given $n$ iid observations $X_1,\ldots,X_n$ with common measure
$\mu$, we propose a procedure to test for whether or not $\mu$
is symmetric.
The bootstrap approach from Section~\ref{sec:symcomp}
for estimating the empirical Wasserstein distance is applied, 
and a permutation test is applied to the bootstrapped sample.
Note that while the Wasserstein-2 metric is specifically used
in our improved symmetrization inequality, for this test,
any Wasserstein-$p$ metric can be utilized as is done in the
numerical simulations below.

The bootstrap-permutation test proceeds as follows:
\begin{enumerate}
  \setcounter{enumi}{-1}
  \item 
  Choose a number $r$ of bootstrap replications to perform.
  \item 
  For each bootstrap replication, permute the data
  by some uniformly randomly drawn 
  $\rho \in S_n$, the symmetric group on $n$ elements.
  \item 
  Use the Hungarian algorithm to compute the optimal
  assignment cost, $\omega_0$, between the data sets 
  $\{X_{\rho(1)},\ldots,X_{\rho(n/2)}\}$
  and $\{-X_{\rho(n/2+1)},\ldots,-X_{\rho(n)}\}$.
  \item 
  Denote this new half-negated data set $Y$ where $Y_i = X_{\rho(i)}$
  for $i\le n/2$ and $Y_i = -X_{\rho(i)}$ for $i>n/2$.
  \item
  Draw $m$ random permutations $\rho_1,\ldots\rho_m \in S_n$. 
  For each $\rho_i$, compute $\omega_i$, the optimal 
  assignment cost between 
  $\{Y_{\rho_i(1)},\ldots,Y_{\rho_i(n/2)}\}$
  and $\{Y_{\rho_i(n/2+1)},\ldots,Y_{\rho_i(n)}\}$.
  \item 
  Return the p-value, $p_j = \#\{\omega_i > \omega_0 \}/m$.
  \item 
  Average the $r$ p-values to get an overall p-value,
  $p = r^{-1}\sum_{j=1}^r p_j$.
\end{enumerate}

Note that for very large data sets, it may be computationally
impractical to find a perfect matching between two sets of
$n/2$ nodes as performing this test as stated has a computational
complexity of order $O(mn^3)$.  In that case, randomly draw $n'<n$ 
elements from
the data set in step~1, draw a $\rho\in S_{n'}$, and proceed as before.

This permutation test was applied to simulated 
multivariate Rademacher data in $\real^5$.  
For sample sizes $n=10$ and $n=100$, let $X_1,\ldots,X_n$
be iid multivariate Rademacher$(p)$ random variables 
where each $X_i$ is comprised of a vector of independent
univariate Rademacher$(p)$ random variables.  For values
of $p \in [0.5,0.8]$, the power of this test was experimentally
computed over 1000 simulations.  The results are displayed
in Figure~\ref{fig:symTestRad}.  For the $\ell^1$ and $\ell^2$ metrics
and Wasserstein distances $W_1$ and $W_2$, the performances 
of the permutation test were comparable except for the $(\ell^2,W_2)$
case, which performed poorer in both the large and small sample size
settings.  For the large sample size, $n=100$, Mardia's test 
for multivariate skewness \citep{MARDIA1970,MARDIA1974} was 
included, which uses the result that
$$
  \frac{6}{n}\sum_{i=1}^n\sum_{j=1}^n 
  \left[
    \TT{(X_i-\bar{X})}\hat{\Sigma}^{-1}(X_j-\bar{X})
  \right]^3 
  \convd
  \distChiSquared{k(k+1)(k+2)/6}
$$
where $\hat{\Sigma}$ is the empirical covariance matrix of 
the data.
However, 
this is shown to be less powerful than the proposed permutation test.
Furthermore, as this test is asymptotic in design, 
it gave erroneous results in the
$n=10$ case and was thus excluded from the figure.

\begin{figure}
  \centering
  \includegraphics[width=0.49\textwidth]{\PICDIR/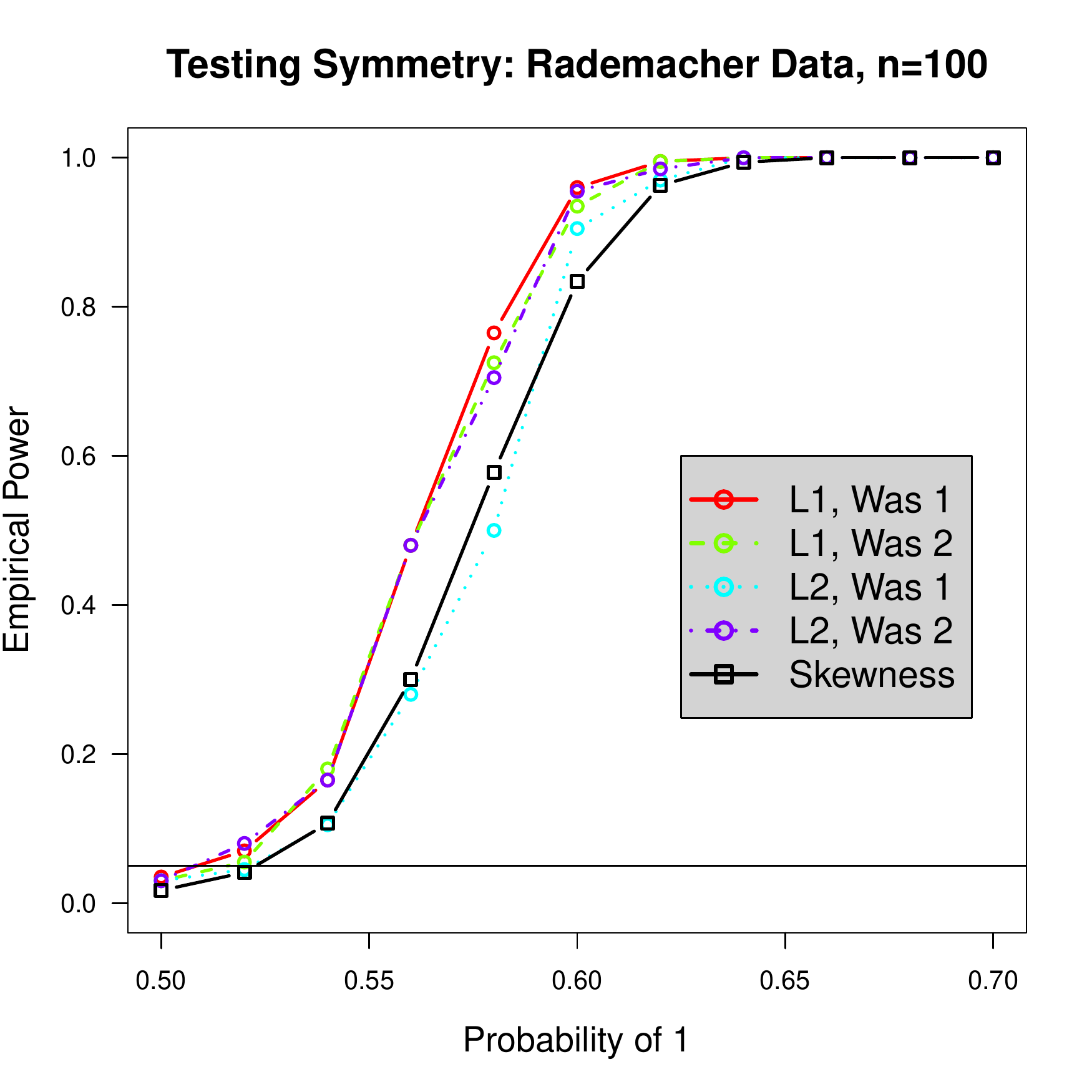}
  \includegraphics[width=0.49\textwidth]{\PICDIR/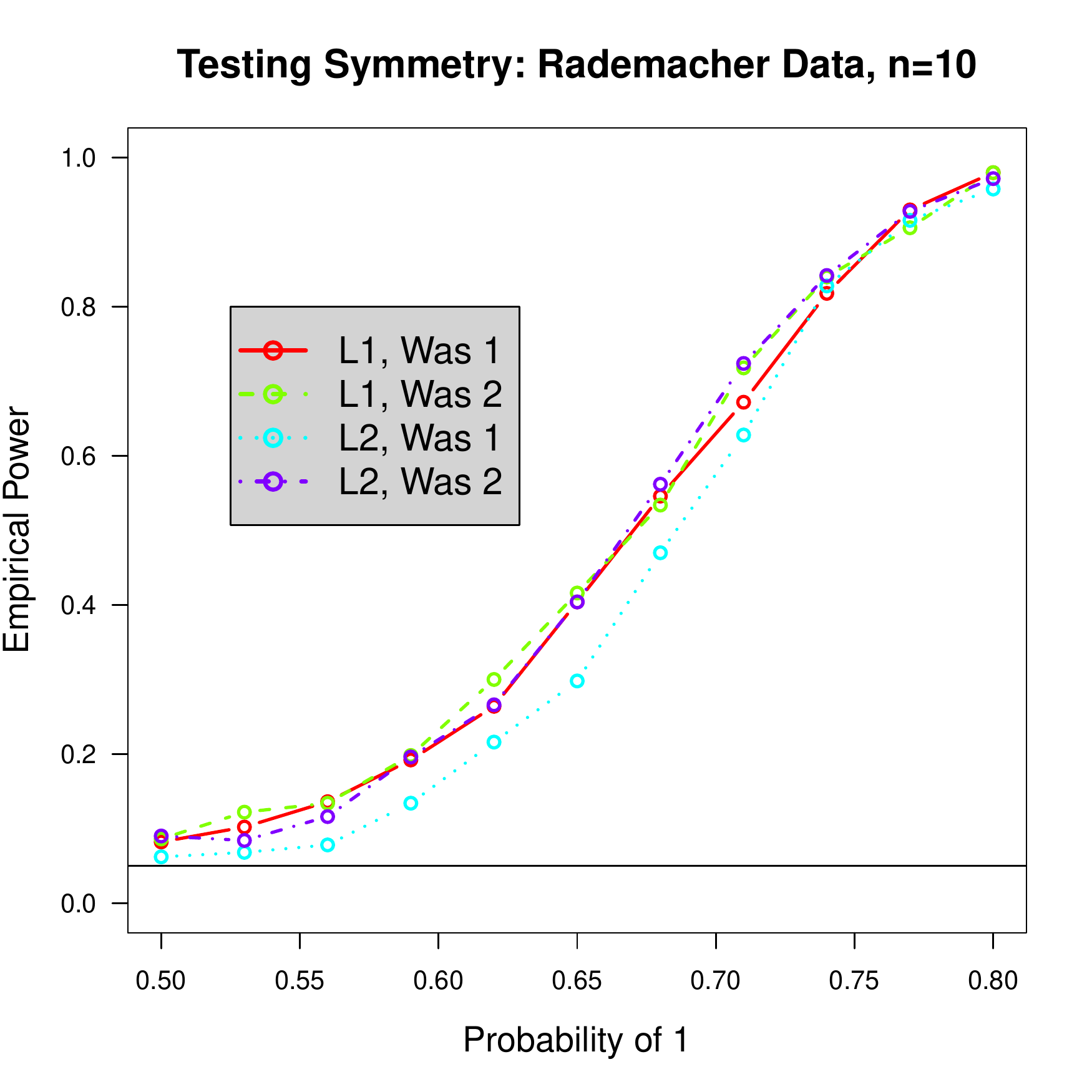}
  \capt{
    \label{fig:symTestRad}
    For data in $\real^5$, the $\ell^1$ and $\ell^2$ metrics, 
    and the Wasserstein distances
    $W_1$ and $W_2$, the experimentally computed power of the
    permutation test is plotted for Rademacher$(p)$ data as $p$, the
    probability of 1, increases thus skewing the distribution.
    The sample size is $n=100$ on the left plot and is $n=10$
    on the right plot.  The $n=100$ case includes an asymptotic
    test for skewness.  This test fails in the nonasymptotic $n=10$ case 
    and thus is not included.
  }
\end{figure}

\subsection{High dimensional confidence sets}
A method for constructing 
nonasymptotic confidence regions for high dimensional data
using a generalized bootstrap procedure was proposed in
the article of \cite{ARLOT2010}.
Beginning with a sample of \iid $Y_1,\ldots,Y_n\in\real^K$
and
the assumptions that the $Y_i$ are 
symmetric about their mean, i.e. $Y_i-\mu\eqdist \mu-Y_i$,
and are bounded in $L_p$-norm, i.e. $\norm{Y_i-\mu}_p\le M$,
they prove, among many other results, that for some 
fixed $\alpha\in(0,1)$, the following holds with probability
$1-\alpha$:
$$
  \phi\left(\bar{Y}-\mu\right) \le \left(\frac{n}{n-1}\right)
  \xv_\veps \phi\left(
    \frac{1}{n}\sum_{i=1}^n \veps_i(Y_i-\bar{Y})
  \right) +
  \frac{2M}{\sqrt{n}}\sqrt{\log(1/\alpha)}
$$
where $\phi:\real^K\rightarrow\real$ is a function that 
is subadditive, positive homogeneous, and bounded by $L_p$-norm.
By substituting our Theorem~\ref{thm:symovern} for their
Proposition~2.4 allows us to drop the symmetry condition 
and achieve a more general $(1-\alpha)$ confidence region.
\begin{prop}
  For a fixed $\alpha\in(0,1)$ and $p\in[1,\infty]$, let
  $\phi:\real^k\rightarrow\real$ be subadditive, positive
  homogeneous, and bounded in $L_p-$norm.  Then, for some
  $M>0$, the following holds with probability at least $1-\alpha$.
  %\begin{multline*}
  $$
    \phi\left(\bar{Y}-\mu\right) \le 
    \xv_\veps \phi\left(
      \frac{1}{n}\sum_{i=1}^n \veps_i(Y_i-\bar{Y})
    \right) +%\\+
    (2n)^{-1/2}\left(
      {2\sqrt{2}M}\sqrt{\log(1/\alpha)} +
      W_2(\mu,\mu^-) 
    \right).
  $$
  %\end{multline*}
\end{prop}

\subsection{Bounds on empirical processes}

Symmetrization arises when bounding the variance 
of an empirical process.
In \cite{BOUCHERON2013}, the following result 
is stated as Theorem 11.8 and is subsequently proved 
using the original symmetrization inequality resulting
in suboptimal coefficients.

\begin{thm}[\cite{BOUCHERON2013}, Theorem 11.8]
  For $i \in \{1,\ldots,n\}$ and $s\in \mathcal{T}$, a
  countable index set, let $X_i = (X_{i,s})_{s\in\mathcal{T}}$
  be a collection of real valued random variables.  Furthermore, 
  let $X_1,\ldots,X_n$ be independent.  
  Assume $\xv X_{i,s}=0$ and $\abs{X_{i,s}}\le1$ for all $i=1,\ldots,n$
  and for all $s\in\mathcal{T}$.  Defining
  $Z=\sup_{s\in\mathcal{T}}\sum_{i=1}^n X_{i,s}$, then
  $$
    \var{Z} \le 8\xv Z + 2 \sigma^2
  $$
  where $\sigma^2=\sup_{s\in\mathcal{T}}\sum_{i=1}^n \xv X_{i,s}^2$.
\end{thm}

The given proof uses the symmetrization inequality twice 
as well as the contraction inequality 
(see \cite{LEDOUXTALAGRAND1991} Theorem 4.4, and 
\cite{BOUCHERON2013} Theorem 11.6)
to establish the bounds
$$
  \xv \sup_{s\in\mathcal{T}}\sum_{i=1}^n X_{i,s}^2
  \le \sigma^2 +
  2 \xv \sup_{s\in\mathcal{T}} \sum_{i=1}^n \veps_i X^2_{i,s}
  ~~\text{ and }~~
  \xv \sup_{s\in\mathcal{T}} \sum_{i=1}^n \veps_i X^2_{i,s}
  \le 4\xv Z.
$$
Making use of the improved symmetrization inequality cuts the
coefficient of $\xv Z$ by a factor of 4 to the tighter 
$$
  \var{Z} \le 2\xv Z + 2\sigma^2 + O(\sqrt{n}).
$$

Beyond this textbook example of bounding the variance of an empirical process, 
symmetrization arguments are used to construct confidence 
sets for empirical processes in 
\cite{GINENICKL2010ADAPTIVE,LOUNICINICKL2011,KERKYACHARIAN2012,FANPARTIII}.
The coefficients in all of their results can be similarly improved 
using the improved symmetrization inequality.

\subsection{Type, Cotype, and Nemirovski's Inequality}

In the probability in Banach spaces setting, 
let $X_i\in(B,\norm{\cdot})$ for $i=1,\ldots,n$ be a collection
of independent mean zero Banach space valued random variables.  
A collection 
of results referred to as \textit{Nemirovski inequalities} 
\citep{NEMIROVSKI2000,DUMBGEN2010}
are concerned
with whether or not there exists a constant $K$ depending only on the 
norm such that 
$$
  \xv\norm*{\sum_{i=1}^n X_i}^2 \le K\sum_{i=1}^n\norm{X_i}^2.
$$
For example, in the Hilbert space setting, 
orthogonality allows for 
$K=1$ and the inequality can be replaced by an equality.

One such result requires the notion of type and cotype.
A Banach space
$(B,\norm{\cdot})$ is said to be of \textit{Rademacher type $p$} 
for $1\le p<\infty$ (respectively, of \textit{Rademacher cotype $q$}
for $1\le q<\infty$)
if there exists a constant $T_p$ (respectively, $C_q$) such that for all 
finite non-random sequences $(x_i)\in B$ and $(\veps_i)$,
a sequence of independent Rademacher random variables,
$$
  \xv\norm*{\sum_i\veps_i x_i}^p \le 
  T_p^p\sum_i\norm{x_i}^p,~~
  \left(\text{respectively, }
    \sum_i\norm{x_i}^q\le
    C_q^{-q} \xv\norm*{\sum_i\veps_i x_i}^q 
  \right).
$$
These definitions and the original symmetrization inequality 
lead to the following proposition.
\begin{prop}[\cite{LEDOUXTALAGRAND1991} Proposition~9.11, 
\cite{DUMBGEN2010} Proposition~3.1]
  \label{prp:nemiOld}
  Let $X_i\in B$ for $i=1,\ldots,n$ and $S_n=n^{-1}\sum_{i=1}^nX_i$.
  If $(B,\norm{\cdot})$ is of type $p\ge1$ with constant $T_p$
(respectively, of cotype $q\ge1$ with constant $C_q$), then
$$
  \xv\norm{S_n}^p \le (2T_p)^pn^{-p} \sum_{i=1}^n \xv\norm{X_i}^p,~~
  \left(
    \xv\norm{S_n}^q \ge (2C_q)^{-q}n^{-q} \sum_{i=1}^n \xv\norm{X_i}^q
  \right)
$$
\end{prop}
The proposition can be refined by applying our improved 
symmetrization inequality along with the Rademacher type $p$
condition if the $X_i$ are additionally norm bounded.  
If the $X_i$ have a common law $\mu$, let 
$W_2 = W_2(\mu,\mu^-)$ be the Wasserstein distance between $\mu$
and its reflection.  

\begin{prop}
  \label{prp:nemiNew}
  Under the setting of Proposition~\ref{prp:nemiOld},
  additionally assume that $\norm{X_i}\le1$ for $i=1,\ldots,n$.
  Then,
  $$
  \xv\norm{S_n}^p \le T_p^pn^{-p} \sum_{i=1}^n \xv\norm{X_i}^p
  + \frac{pW_2}{\sqrt{2n}},~~
  \left(
    \xv\norm{S_n}^q \ge C_q^{-q}n^{-q} \sum_{i=1}^n \xv\norm{X_i}^q
    - \frac{qW_2}{\sqrt{2n}}
  \right)
  $$
\end{prop}
\begin{proof}
  In the context of Theorem~\ref{thm:symovern}, set 
  $\psi(\cdot) = \norm{\cdot}^p$.  Given the bound $\norm{X_i}\le1$,
  we have that $\norm{\psi}_{Lip}=p$.  Scale by $p$, and the first 
  result follows.
\end{proof}

Note that for identically distributed $X_i\in B$, 
the order of the original bound for a type $p$ Banach space 
is $O(n^{1-p})$ while the Wasserstein correction term is $O(n^{-1/2})$.
This correction will give an obvious benefit for spaces of type $p<3/2$.  
However, even for spaces 
of type 2, the new bound can be tighter specifically in the 
high dimensional setting when $d\gg n$.  Indeed, consider
$\ell_\infty(\real^d)$, which is discussed in particular 
in Section~3.2 of 
\cite{DUMBGEN2010} where it is shown
to be of type 2 with constant $T_p=\sqrt{2\log(2d)}$.  For 
iid $X_i\in\ell_\infty(\real^d)$, the bounds to compare
are
$$
  \frac{8\log(2d)}{n} \xv\norm{X_i}_\infty^2
  ~~~\text{ and }~~~
  \frac{2\log(2d)}{n} \xv\norm{X_i}_\infty^2 + 
  \sqrt{\frac{2}{n}}W_2(\mu,\mu^-).
$$
Figure~\ref{fig:nemiBeta} displays such a comparison for 
$n=10$, $d\in\{5,25,50\}$, and iid 
$X_{i,j}+\alpha/(1+\alpha)\dist\distBeta{\alpha}{1}$
for $i=1,\ldots,n$ and $j=1,\ldots,d$.  
Hence, the $X_i$ are Beta random variables that are shifted 
to have zero mean.
$W_2(\mu,\mu^-)$ is approximated
by $\xv W_2(\mu_5,\mu_5^-)$, which is computed via the bootstrap
procedure outlined in Section~\ref{sec:symcomp}.  
The new bound can be seen to have better performance than the old
one specifically in the cases of $d=25$ and $d=50$ when $\alpha$ 
is not too large.

\begin{figure}
  \begin{center}
    \includegraphics[width=\textwidth]{\PICDIR/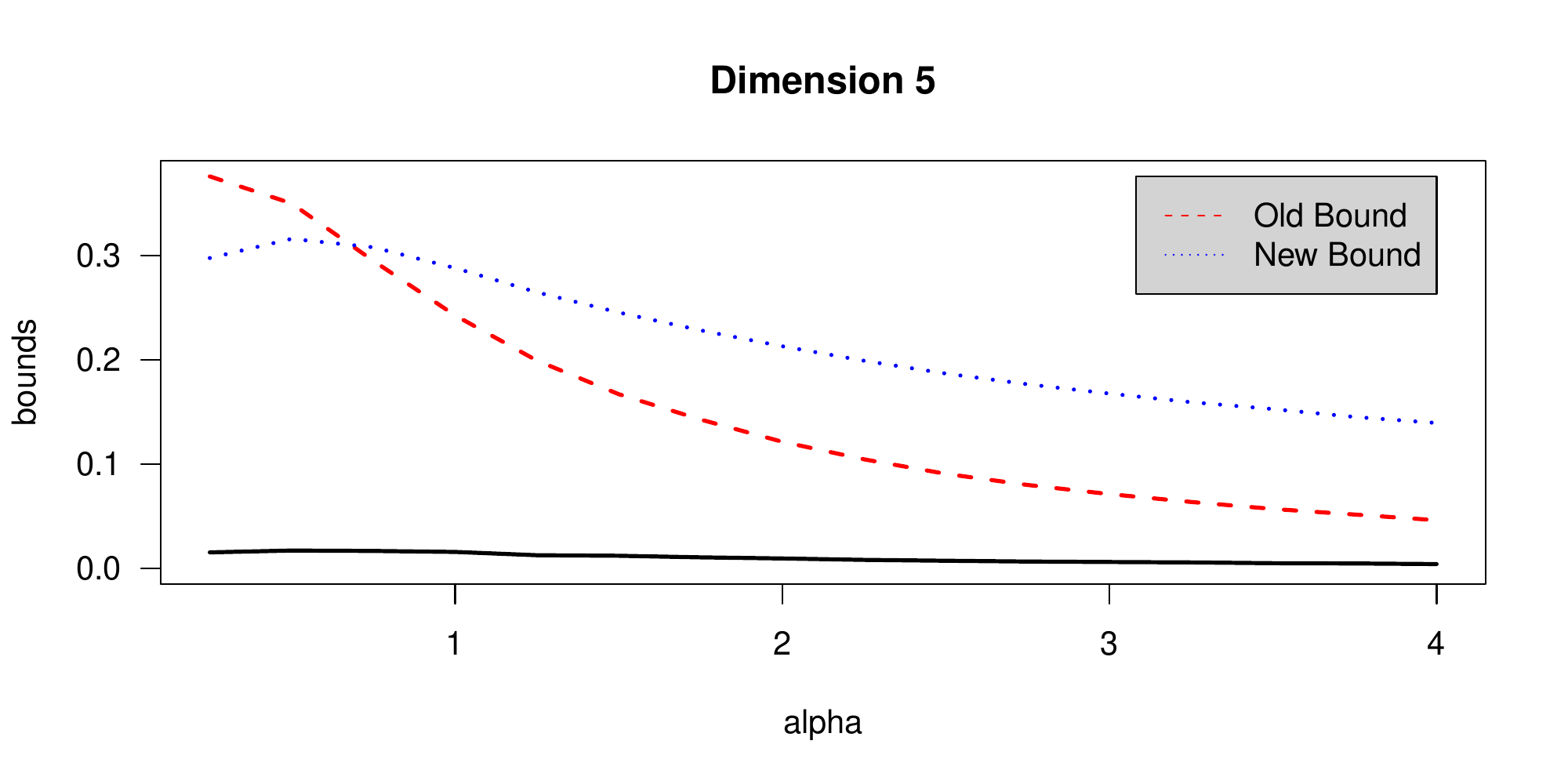}
    \vspace{-0.5in}

    \includegraphics[width=\textwidth]{\PICDIR/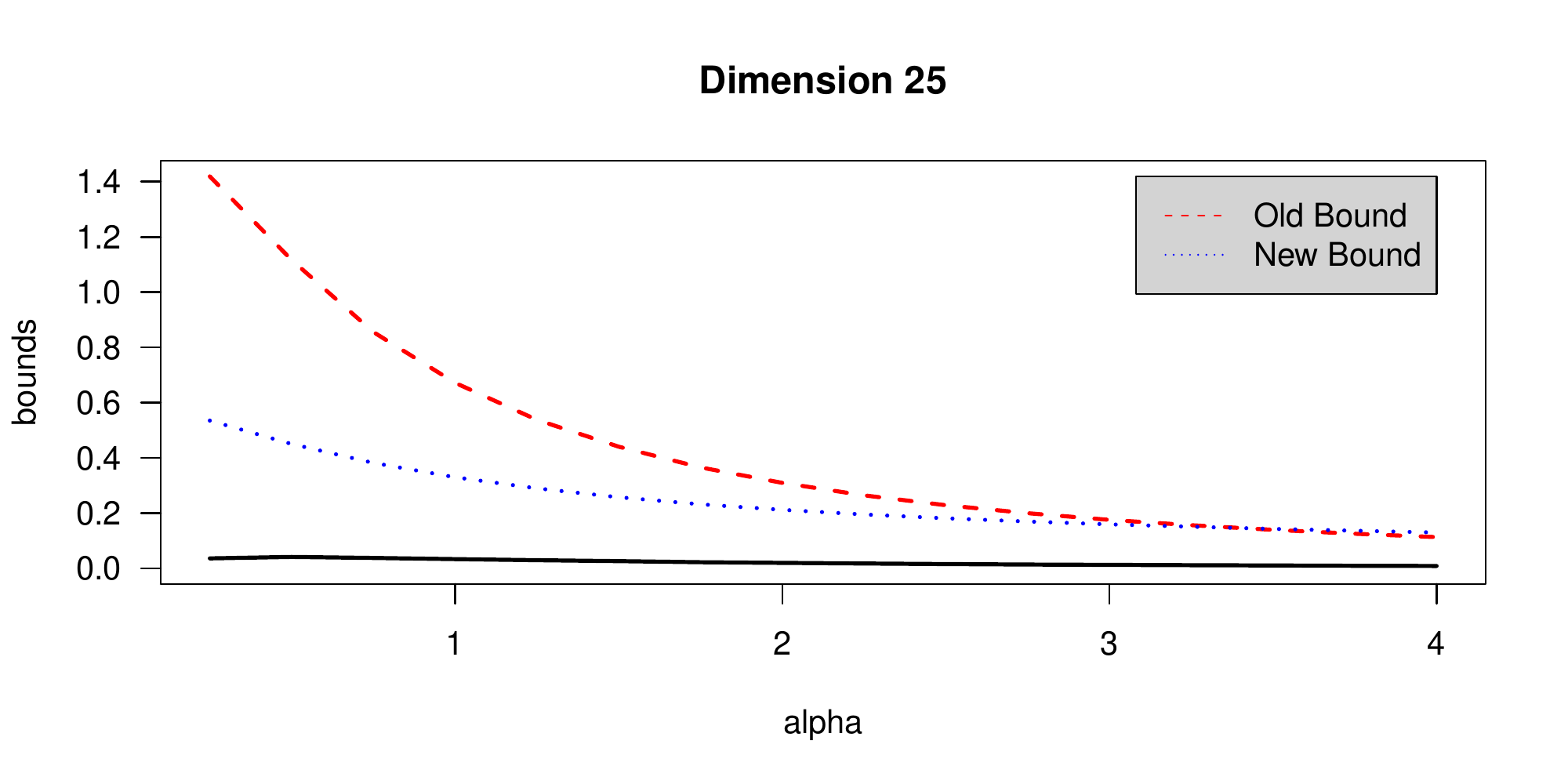}
    \vspace{-0.5in}

    \includegraphics[width=\textwidth]{\PICDIR/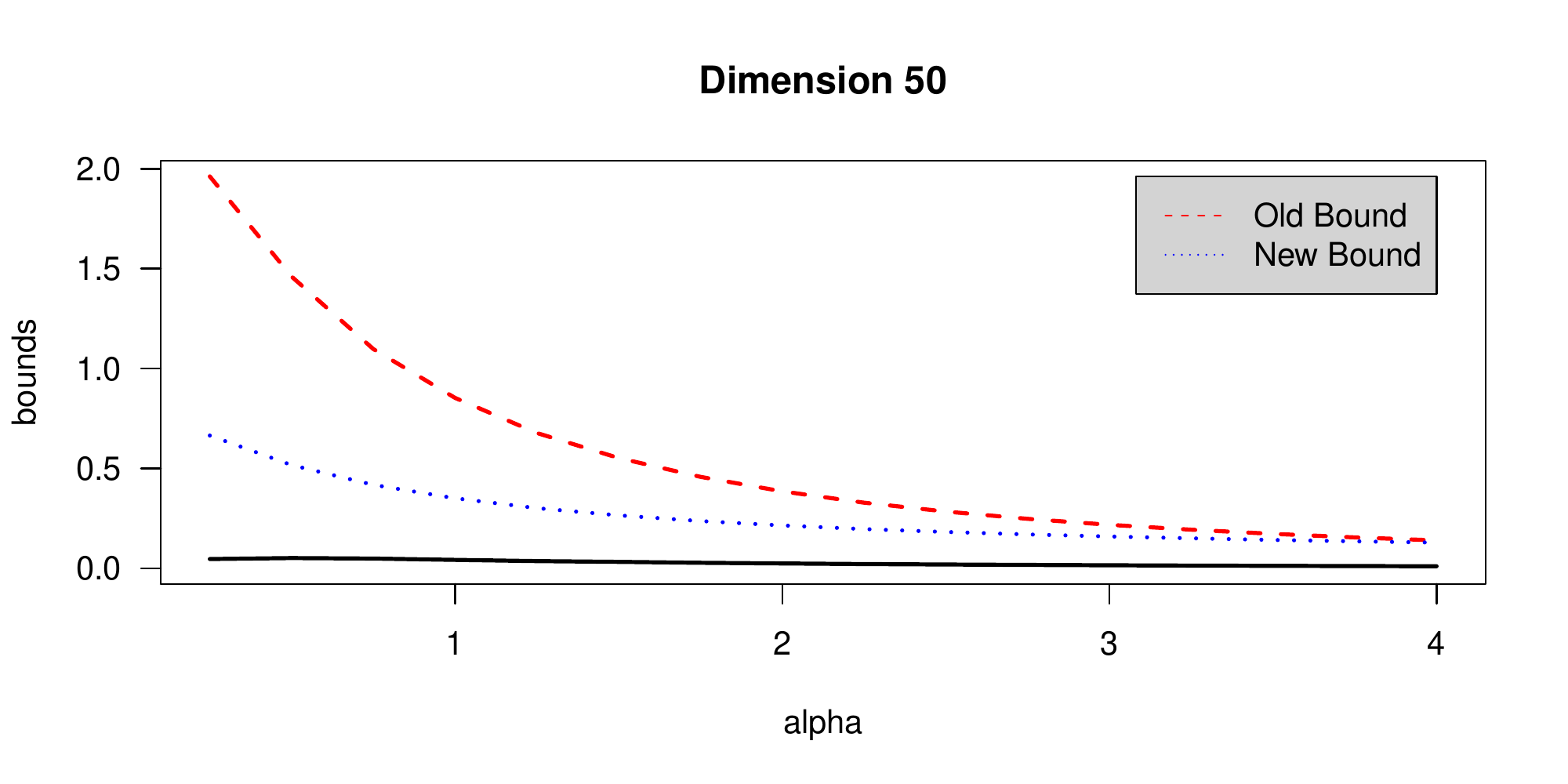}
  \end{center}
  \capt{
    \label{fig:nemiBeta}
    A comparison of the old bound from Proposition~\ref{prp:nemiOld}, 
    the red dashed line,
    and the new bound from Proposition~\ref{prp:nemiNew}, the
    blue dotted line, for
    a sample $n=10$ $X_i\in\ell_\infty(\real^d)$ for dimensions 
    $d\in\{5,25,50\}$.  Each $X_i=(X_{i,1},\ldots,X_{i,d})$ where
    each $X_{i,j}+\alpha/(1+\alpha)\distiid\distBeta{\alpha}{1}$. 
    The solid black line indicates the left
    hand side in the two propositions of $\xv\norm{S_n}_\infty^2$.
  }
\end{figure}

\subsubsection{A Nemirovski variant with weak variance}

As one further example of improved symmetrization,
a variation of 
Nemirovski's inequality found in 
Section 13.5 of \cite{BOUCHERON2013} 
is proved via a similar symmetrization argument
for the $\ell_p$ norm with $p\ge1$.  
Let $X_1,\ldots,X_n\in\real^d$ be independent mean zero 
random variables.  
Let $B_q=\{x\in\real^d:\norm{x}_q\le1\}$, and define the 
weak variance 
$\Sigma_p^2 = n^{-2}\xv\sup_{t\in B_q}\sum_{i=1}^n\iprod{t}{X_i}^2$.
The resulting inequality is
$$
  \xv\norm{S_n}^2_p \le 578 d \Sigma_p^2.
$$
Replacing the old symmetrization inequality with the improved
version reduces the coefficient of 578 roughly by a factor of 4
resulting in
$$ 
  \xv\norm{S_n}^2_p \le 146 d \Sigma_p^2 + O(n^{-1/2}).
$$

\section{Discussion}

The symmetrization inequality is a fundamental result 
for probability in Banach spaces, concentration inequalities,
and many other related areas.  However, not accounting 
for the amount of asymmetry in the given random variables
has led to pervasive powers of two throughout derivative 
results.  Our improved symmetrization inequality 
incorporates such a quantification of asymmetry through 
use of the Wasserstein distance.  Besides being theoretically
sound, it is shown in simulations to provide a tightness 
superior to that of the original result.
Going beyond the 
inequality itself, this Wasserstein distance offers 
a novel and powerful way to analyze the symmetry of random variables 
or lack thereof.  It can and should be applied to countless other 
results that were not considered in this current work.

\appendix
\section{Past results used}

\begin{lm}[
    Kantorovich-Rubinstein Duality,
    see \cite{VILLANI2008OPTTRANS}
  ]
  \label{lem:duality}
  Under the setting of Definition~\ref{def:wasser},
  $$
    W_1(\mu,\nu) = \sup_{\norm{\phi}_{Lip}\le 1}\left\{
      \int_{\mathcal{X}} \phi d\mu -
      \int_{\mathcal{X}} \phi d\nu
    \right\}.
  $$
\end{lm}

\begin{lm}
  \label{lem:order}
  Under the setting of Definition~\ref{def:wasser},
  for $p<q$,
  $$
    W_p(\mu,\nu) \le W_q(\mu,\nu).
  $$
\end{lm}
\begin{proof} Jensen or \holders Inequality \end{proof}

\begin{lm}[
    Convolution property of $W_2$, 
    see \cite{BICKELFREEDMAN1981}
  ]
  \label{lem:conv}
  For Hilbert space valued random variables $X_i$ with law $\mu_i$ 
  and $Y_i$ with law $\nu_i$ for $i=1,\ldots,n$, 
  define $\mu^{*n}$ to be 
  the law of $\sum_{i=1}^n X_i$ and similarly for $\nu^{*n}$.
  Then,
  $$
    W_2^2(\mu^{*n},\nu^{*n}) \le 
    \sum_{i=1}^n W_2^2(\mu_i,\nu_i).
  $$
\end{lm}

\begin{lm}[  
    Convergence of Empirical Measure,
    see \cite{BICKELFREEDMAN1981}
  ] 
  \label{lem:converge}
  Let $X_1,\ldots,X_n$ be \iid Banach space valued random variables
  with common law $\mu$.  Let $\mu_n$ be the empirical distribution
  of the $X_i$.  Then,
  $$
    W_p(\mu_n,\mu) \rightarrow 0,~~\text{as }n\rightarrow\infty.
  $$
\end{lm}

% Bibliography

\bibliographystyle{plainnat}
\bibliography{kasharticle,kashbook}

\begin{thebibliography}{23}
\providecommand{\natexlab}[1]{#1}
\providecommand{\url}[1]{\texttt{#1}}
\expandafter\ifx\csname urlstyle\endcsname\relax
  \providecommand{\doi}[1]{doi: #1}\else
  \providecommand{\doi}{doi: \begingroup \urlstyle{rm}\Url}\fi

\bibitem[Ahuja et~al.(1993)Ahuja, Magnanti, and Orlin]{AHUJA1993NETWORKFLOWS}
Ravindra~K. Ahuja, Thomas~L. Magnanti, and James~B. Orlin.
\newblock \emph{Network Flows: Theory, Algorithms, and Applications}.
\newblock Prentice-Hall, Inc, 1993.

\bibitem[Arlot et~al.(2010)Arlot, Blanchard, and Roquain]{ARLOT2010}
Sylvain Arlot, Gilles Blanchard, and Etienne Roquain.
\newblock Some nonasymptotic results on resampling in high dimension, i:
  confidence regions.
\newblock \emph{The Annals of Statistics}, 38\penalty0 (1):\penalty0 51--82,
  2010.

\bibitem[Bickel and Freedman(1981)]{BICKELFREEDMAN1981}
Peter~J Bickel and David~A Freedman.
\newblock Some asymptotic theory for the bootstrap.
\newblock \emph{The Annals of Statistics}, pages 1196--1217, 1981.

\bibitem[Boucheron et~al.(2013)Boucheron, Lugosi, and Massart]{BOUCHERON2013}
St{\'e}phane Boucheron, G{\'a}bor Lugosi, and Pascal Massart.
\newblock \emph{Concentration inequalities: A nonasymptotic theory of
  independence}.
\newblock Oxford University Press, 2013.

\bibitem[D{\"u}mbgen et~al.(2010)D{\"u}mbgen, van~de Geer, Veraar, and
  Wellner]{DUMBGEN2010}
Lutz D{\"u}mbgen, Sara~A van~de Geer, Mark~C Veraar, and Jon~A Wellner.
\newblock Nemirovski's inequalities revisited.
\newblock \emph{American Mathematical Monthly}, 117\penalty0 (2):\penalty0
  138--160, 2010.

\bibitem[Efron and Stein(1981)]{EFRONSTEIN1981}
Bradley Efron and Charles Stein.
\newblock The jackknife estimate of variance.
\newblock \emph{The Annals of Statistics}, pages 586--596, 1981.

\bibitem[Fan(2011)]{FANPARTIII}
Zhou Fan.
\newblock Confidence regions for infinite-dimensional statistical parameters.
\newblock \emph{Part III essay in Mathematics, University of Cambridge}, 2011.
\newblock \url{http://web.stanford.edu/~zhoufan/PartIIIEssay.pdf}.

\bibitem[Fournier and Guillin(2015)]{FOURNIER2015}
Nicolas Fournier and Arnaud Guillin.
\newblock On the rate of convergence in wasserstein distance of the empirical
  measure.
\newblock \emph{Probability Theory and Related Fields}, 162\penalty0
  (3-4):\penalty0 707--738, 2015.

\bibitem[Gin{\'e} and Nickl(2010)]{GINENICKL2010ADAPTIVE}
Evarist Gin{\'e} and Richard Nickl.
\newblock Adaptive estimation of a distribution function and its density in
  sup-norm loss by wavelet and spline projections.
\newblock \emph{Bernoulli}, 16\penalty0 (4):\penalty0 1137--1163, 2010.

\bibitem[Gin{\'e} and Zinn(1984)]{GINEZINN1984}
Evarist Gin{\'e} and Joel Zinn.
\newblock Some limit theorems for empirical processes.
\newblock \emph{The Annals of Probability}, 12\penalty0 (4):\penalty0 929--989,
  1984.

\bibitem[Kerkyacharian et~al.(2012)Kerkyacharian, Nickl, and
  Picard]{KERKYACHARIAN2012}
Gerard Kerkyacharian, Richard Nickl, and Dominique Picard.
\newblock Concentration inequalities and confidence bands for needlet density
  estimators on compact homogeneous manifolds.
\newblock \emph{Probability Theory and Related Fields}, 153\penalty0
  (1-2):\penalty0 363--404, 2012.

\bibitem[Koltchinskii(2006)]{KOLTCHINSKII2006}
Vladimir Koltchinskii.
\newblock Local rademacher complexities and oracle inequalities in risk
  minimization.
\newblock \emph{The Annals of Statistics}, 34\penalty0 (6):\penalty0
  2593--2656, 2006.

\bibitem[Kuhn(1955)]{KUHN1955HUNGARIAN}
Harold~W Kuhn.
\newblock The hungarian method for the assignment problem.
\newblock \emph{Naval research logistics quarterly}, 2\penalty0 (1-2):\penalty0
  83--97, 1955.

\bibitem[Ledoux and Talagrand(1991)]{LEDOUXTALAGRAND1991}
Michel Ledoux and Michel Talagrand.
\newblock \emph{Probability in Banach Spaces: isoperimetry and processes},
  volume~23.
\newblock Springer, 1991.

\bibitem[Lounici and Nickl(2011)]{LOUNICINICKL2011}
Karim Lounici and Richard Nickl.
\newblock Global uniform risk bounds for wavelet deconvolution estimators.
\newblock \emph{The Annals of Statistics}, 39\penalty0 (1):\penalty0 201--231,
  2011.

\bibitem[Mardia(1970)]{MARDIA1970}
Kanti~V Mardia.
\newblock Measures of multivariate skewness and kurtosis with applications.
\newblock \emph{Biometrika}, 57\penalty0 (3):\penalty0 519--530, 1970.

\bibitem[Mardia(1974)]{MARDIA1974}
Kanti~V Mardia.
\newblock Applications of some measures of multivariate skewness and kurtosis
  in testing normality and robustness studies.
\newblock \emph{Sankhy{\=a}: The Indian Journal of Statistics, Series B}, pages
  115--128, 1974.

\bibitem[Nemirovski(2000)]{NEMIROVSKI2000}
Arkadi Nemirovski.
\newblock Topics in non-parametric.
\newblock \emph{Ecole d’Et{\'e} de Probabilit{\'e}s de Saint-Flour},
  28:\penalty0 85, 2000.

\bibitem[Panchenko(2003)]{PANCHENKO2003}
Dmitry Panchenko.
\newblock Symmetrization approach to concentration inequalities for empirical
  processes.
\newblock \emph{Annals of Probability}, pages 2068--2081, 2003.

\bibitem[Rhee and Talagrand(1986)]{RHEETALAGRAND1986}
WanSoo~T Rhee and Michel Talagrand.
\newblock Martingale inequalities and the jackknife estimate of variance.
\newblock \emph{Statistics \& probability letters}, 4\penalty0 (1):\penalty0
  5--6, 1986.

\bibitem[Steele(1986)]{STEELE1986}
J~Michael Steele.
\newblock An {E}fron-{S}tein inequality for nonsymmetric statistics.
\newblock \emph{The Annals of Statistics}, pages 753--758, 1986.

\bibitem[Steele(1997)]{STEELE1997}
J~Michael Steele.
\newblock \emph{Probability theory and combinatorial optimization}, volume~69.
\newblock Siam, 1997.

\bibitem[Villani(2008)]{VILLANI2008OPTTRANS}
C{\'e}dric Villani.
\newblock \emph{Optimal transport: old and new}, volume 338.
\newblock Springer Science \& Business Media, 2008.

\end{thebibliography}

%%%%%%%%%%%%%%%%
\end{document}